\documentclass[12pt]{article}
\usepackage{amsfonts}
\usepackage[margin=1.00in]{geometry}
\usepackage{amsmath,amssymb,amsthm,mathtools}
\usepackage{microtype}
\usepackage{enumitem}
\usepackage{booktabs,array}
\usepackage{authblk}
\usepackage[
  colorlinks=true,
  linkcolor=red,
  citecolor=blue,
  urlcolor=blue
]{hyperref}
\usepackage{mathrsfs}
\allowdisplaybreaks 

\newtheorem{theorem}{Theorem}[section]
\newtheorem{lemma}[theorem]{Lemma}
\newtheorem{proposition}[theorem]{Proposition}
\newtheorem{corollary}[theorem]{Corollary}
\theoremstyle{definition}

\theoremstyle{remark}

\newcommand{\B}{\mathbb B}
\newcommand{\PP}{\mathbb P}
\newcommand{\CC}{\mathbb C}

\newcommand{\cI}{\mathcal I}
\newcommand{\cK}{\mathcal K}

\newcommand{\kcone}{\mathbf K_{\mathrm{cone}}}
\newcommand{\rk}{\operatorname{rank}}
\newcommand{\im}{\operatorname{im}}
\newcommand{\Span}{\operatorname{Span}}

\newcommand{\ord}{\operatorname{ord}}
\newcommand{\Div}{\operatorname{div}}
\newcommand{\Dop}{\mathscr D}

\newcommand{\cO}{\mathcal O}
\newcommand{\md}{\mathbf m}

\newcommand{\BB}{{\mathbb B}}

\newcommand{\C}{\mathbb C}

\newcommand{\HH}{\mathbb H}

\newcommand{\ov}{\overline}
\newcommand{\p}{\partial}

\numberwithin{equation}{section}

\title{Rigidity for proper holomorphic ball maps with degenerate CR Gauss map}

\author{Tianzhi Hu, Wanke Yin{\footnote{Supported by Hubei Provincial Innovation Group Project (No. 2025AFA044)}}\ \  and Pingsan Yuan}

\date{}

\begin{document}
\maketitle

\begin{abstract}
Let $n$ and $N$ be integers with $2\leq n<N$, and let
$F:\mathbb{B}^n\to\mathbb{B}^N$ be a proper holomorphic map that
admits a $C^{N-n+1}$-smooth extension to the boundary. We prove that
if the CR Gauss map of its boundary restriction is generically
degenerate, then $F$ is holomorphically equivalent, under
composition with automorphisms of the source and target balls, to
$V_m\oplus 0$ for some positive integer $m$, where $
V_m(z)=\big(\sqrt{\frac{m!}{\alpha!}} z^\alpha\big)_{|\alpha|=m} $
is the degree-$m$ Veronese map. This resolves a problem posed by
Xiaojun Huang.
\end{abstract}
\addtocontents{toc}{\protect\setcounter{tocdepth}{2}}

\section{Introduction}

The Gauss map records how the tangent spaces of a submanifold vary.
For an oriented regular surface $M\subset\mathbb R^3$, the classical
Gauss map is the unit normal map $\nu:M\to S^2$. Under the natural
identification of the tangent planes, the shape operator is $-d\nu$,
and its determinant is the Gaussian curvature. Degeneracy of the
Gauss map thus detects tangent directions along which the normal has
vanishing first-order variation. In the case of surfaces in $\mathbb
P^3$, Segre \cite{Seg1910} showed that a surface with degenerate Gauss map is
either a cone over a curve or the tangential variety to a curve.

The corresponding construction in projective geometry assigns to
each smooth point its embedded projective tangent space. Let
$X\subset\mathbb P^N$ be an irreducible complex projective variety
of dimension $k$, with smooth locus $X_{\mathrm{reg}}$. Its
projective Gauss map is
\begin{align}\label{eq:projGauss}
\Gamma_X:X_{\mathrm{reg}} \longrightarrow
\operatorname{Gr}(k+1,\mathbb C^{N+1}), \qquad x\longmapsto \widehat
T_xX,
\end{align}
where $\widehat T_xX\subset\mathbb C^{N+1}$ is the
$(k+1)$-dimensional vector subspace whose projectivization is the
embedded projective tangent $k$-plane $\mathbb T_xX\subset\mathbb
P^N$. The Gauss map is called degenerate if its generic fibers have
positive dimension, equivalently, if its generic differential rank
is less than $k$. Its differential is governed by the projective
second fundamental form.

The geometry of varieties with degenerate Gauss maps has been
studied extensively. Griffiths--Harris \cite{GH79} developed a
systematic approach using moving frames and fundamental forms. In
particular, they showed that the closure of a general fiber of the
projective Gauss map is a projective linear subspace. Thus, a
variety with degenerate Gauss map is locally ruled by linear spaces
along which its embedded tangent space remains fixed. The
differential-geometric structure of submanifolds with degenerate
Gauss maps was further investigated by Akivis--Goldberg
\cite{AG01}; see also their monograph \cite{AG04} for a systematic
treatment of varieties with degenerate Gauss maps. The projective
differential geometry of the second fundamental form and its
relation to tangential and secant defects was further developed by
Landsberg \cite{La94,La96}. More specialized developments include
the study of developable varieties of Gauss rank $2$ by Piontkowski
\cite{Pi02}. Related questions have also been investigated by Ran
\cite{Ra81} for subvarieties of abelian varieties and by Hwang
\cite{Hw02} in complex hyperbolic space forms.

In CR geometry, the work of Chern--Moser \cite{CM74} and Webster
\cite{We78} provides a fundamental framework for the study of CR
mappings. The role of second fundamental forms in rigidity questions
was further developed by Ebenfelt--Huang--Zaitsev
\cite{EHZ04,EHZ05}. Ji--Yuan \cite{JY10} proved that a smooth CR
embedding of a strictly pseudoconvex hypersurface into a sphere
whose CR second fundamental form vanishes has image locally
contained in a totally geodesic sphere. Cheng--Ji \cite{CJ12}
obtained related linearity results under conditions on second
fundamental forms. Cheng--Ji--Liu \cite{CJL13} investigated Gauss
maps of spherical CR submanifolds, while Yin--Yuan--Zhang \cite{YYZ20}
established linearity criteria under a restriction on geometric
rank, together with additional hypotheses on degeneracy invariants
or the target dimension.

These developments motivate the classification of proper holomorphic
maps between balls whose boundary Gauss maps are generically
degenerate. In this paper, we address a problem posed by Xiaojun
Huang and show that, under the boundary regularity assumption stated
below, such maps are equivalent to homogeneous Veronese maps with
possible zero components added. The precise definition of the CR
Gauss map used here is given in Section~2; generic degeneracy refers
to positive-dimensional generic real fibers of the boundary Gauss
map.

Two proper holomorphic maps $f,g:\mathbb B^n\to\mathbb B^N$ are said
to be \emph{holomorphically equivalent} if $ g=\tau\circ
f\circ\sigma $ for some $\sigma\in\operatorname{Aut}(\mathbb B^n)$
and $\tau\in\operatorname{Aut}(\mathbb B^N)$. With this notion of
equivalence, we can now state our main theorem.

\begin{theorem}\label{thm:main}
Let $n$ and $N$ be integers with $2\le n<N$, and let $F:\mathbb
B^n\to\mathbb B^N$ be a proper holomorphic map that admits a
$C^{N-n+1}$-smooth extension to $\overline{\mathbb B^n}$. If the CR
Gauss map of $F|_{\partial\mathbb B^n}$ is generically degenerate,
then $F$ is holomorphically equivalent to $V_m\oplus 0$ for some
positive integer $m$ satisfying
\[
M_m:=\binom{n+m-1}{m}\le N,
\]
where
\[
V_m(z)=
\left(\sqrt{\frac{m!}{\alpha!}}\,z^\alpha\right)_{|\alpha|=m}
:\mathbb B^n\longrightarrow\mathbb B^{M_m}
\]
is the degree-$m$ Veronese map, and the zero component has dimension
$N-M_m$.
\end{theorem}

The boundary regularity assumption allows us to apply
Forstneri\v{c}'s celebrated rationality theorem \cite{Fo89}. Together with the
boundary extension result of Cima--Suffridge \cite{CS90}, this
reduces the problem to the study of rational maps that extend
holomorphically across the boundary.

The Veronese maps provide the nonlinear models in the theorem. The
multinomial identity gives
\[
\|V_m(z)\|^2=\|z\|^{2m},
\]
so $V_m$ maps the unit ball properly into $\mathbb B^{M_m}$. For
$z\ne0$, write $ E_z=\operatorname{im}DV_m(z). $ Homogeneity implies
\[
DV_m(tz)=t^{m-1}DV_m(z),\qquad t\in\mathbb C^*,
\]
and hence $E_{tz}=E_z$. Euler's identity $ DV_m(z)z=m\,V_m(z)$ also
shows that $V_m(z)\in E_z$. Consequently, the affine tangent plane
$V_m(z)+E_z$ equals $E_z$ and remains fixed along each punctured
complex radial line. Its projective closure in $\mathbb P^{M_m}$
remains fixed as well. Restricting to the boundary, each Hopf circle
$ \{e^{i\theta}q:0\le\theta<2\pi\}$ of $q\in\partial\mathbb B^n$ is
mapped to a single tangent plane under the Gauss map. 
For $m\ge2$, by \eqref{eq:CRGauss}, the CR Gauss map of $V_m$ is
\[
\gamma_{V_m}(z)=\operatorname{im}DV_m(z)=E_z,
\qquad z\in\partial\mathbb B^n.
\]
A direct computation shows that the only complex ambient directions
along which $E_z$ remains fixed to first order form the radial line
$\mathbb Cz$. Hence
\[
\ker_{\mathbb R}d\gamma_{V_m,z}
=
\mathbb Cz\cap T_z\partial\mathbb B^n
=
\mathbb R\, iz.
\]
Thus the generic real fibers are one-dimensional. When $m=1$, the CR Gauss map is constant.

Our result also contributes to the classification of proper
holomorphic maps between balls. The first rigidity theorem,
developed through the work of Webster \cite{We79}, Faran
\cite{Fa86}, and Huang \cite{Hu99}, states that a proper holomorphic
map $\mathbb B^n\to\mathbb B^N$ with $n\ge3$ and $N<2n-1$ is totally
geodesic if it is $C^2$-smooth up to the boundary. For rational
proper maps, Huang--Ji \cite{HJ01} classified the case $N=2n-1$,
$n\ge3$, obtaining the linear and Whitney maps. For $n\ge4$, Hamada
\cite{Ha05} classified the case $N=2n$, and Huang--Ji--Xu
\cite{HJX06} established the second gap theorem for $2n<N<3n-3$. In
this range, the models are the D'Angelo family \cite{DA88},
\[
F_\theta(z',z_n) = \bigl(z',(\cos\theta)z_n,
(\sin\theta)z_nz',(\sin\theta)z_n^2\bigr), \qquad
0\le\theta\le\frac{\pi}{2},
\]
with possible zero components added, where
$z'=(z_1,\ldots,z_{n-1})$. At $N=3n-3$ and $n\ge4$,
Andrews--Huang--Ji--Yin \cite{AHJY16} completed the rational
classification, with the generalized Whitney map appearing as an
additional model.

In source dimension two, Faran \cite{Fa82} classified rational
proper maps from $\mathbb B^2$ to $\mathbb B^3$ into the following
four inequivalent classes: $ (z,w,0),\ (z,zw,w^2),\
(z^2,\sqrt{2}\,zw,w^2),\  (z^3,\sqrt{3}\,zw,w^3). $ Beyond these
special cases, the classification problem becomes substantially more
complicated. Theorem~\ref{thm:main} provides a classification under
the geometric hypothesis of Gauss map degeneracy without imposing an
upper bound on the target dimension. In particular, since $M_m\ge
M_2=n(n+1)/2$ for $m\ge2$, the theorem implies linearity whenever
\[
n<N<\frac{n(n+1)}{2}.
\]
Moreover, by Huang--Ji--Xu \cite[Corollary 1.4]{HJX05}, the boundary regularity  can be relaxed to $C^3$ for $N$ in this range.
The quadratic Veronese map attains the threshold $N=n(n+1)/2$.

We conclude by outlining the proof. Our approach uses the algebraic
structure of the polarized sphere identity and the geometry of Gauss
fibers. Section~\ref{sec:ball-to-projective} develops the polarized sphere identity and derives
a differential identity associated with a direction in the kernel of
the affine Gauss map. Section~\ref{sec:kernel-divisor} analyzes the factors of the
polarized multiplier by polynomial differentiation and establishes
the resulting branch alternatives. Section~\ref{sec:classification} excludes the mixed
branch and proves linearity in the noncharacteristic case. In the
remaining pure case, first-order identities identify a center inside
the source ball, and radial rigidity yields the Veronese normal
form.

\section{From proper holomorphic ball maps to projective Gauss degeneracy}\label{sec:ball-to-projective}

\subsection{Preliminaries and the homogeneous sphere multiplier}

Let $n,N$ be integers with $2\leq n<N$, and let
\[
F:\B^n\longrightarrow\B^N
\]
be a proper holomorphic map of class $C^{N-n+1}$ up to the boundary.
By Forstneri\v{c}'s rationality theorem \cite[Theorem~1.4]{Fo89},
$F$ is rational. The boundary extension theorem of Cima--Suffridge
\cite{CS90} then implies that $F$ extends holomorphically to a
neighborhood of $\overline{\B^n}$. The boundary regularity
assumption is used only to obtain this reduction. Write
\[
F=\frac{p}{q}, \ \ f:=F|_{\p \BB^n},
\]
where $q$ and the components of the polynomial map $p:\CC^n\to\CC^N$
have no common nonconstant polynomial factor. By
\cite[Theorem~1.4]{Fo89}, we have
\begin{lemma}\label{lem:nopole}
$q$ is nonvanishing in a neighborhood of $\overline{\B^n}$.
\end{lemma}

For $Z=(Z_0,\ldots,Z_n)$ and $W=(W_0,\ldots,W_n)$ in $\CC^{n+1}$,
set
\[
H(Z):=-|Z_0|^2+\sum_{j=1}^n|Z_j|^2, \qquad
h(Z,W):=-Z_0W_0+\sum_{j=1}^n Z_jW_j.
\]
Thus $h$ is the complex bilinear polarization of $H$ with
$H(Z)=h(Z,\overline Z)$. Similarly, for $U,V\in\CC^{N+1}$, set
\[
B(U,V):=-U_0V_0+\sum_{j=1}^N U_jV_j.
\]
For a polynomial $R(Z)=\sum_I a_I Z^I$, define
\[
R^\#(W):=\sum_I\overline{a_I}W^I.
\]
We apply this notation componentwise to polynomial maps. For a
polynomial $S(Z,W)$, define
\[
S^\#(Z,W):=\overline{S(\overline W,\overline Z)}.
\]
In particular, $h^\#=h$ and $B^\#=B$.

Consider the incidence variety
\[
\mathcal{I}:=\bigl\{([Z],[W])\in\PP^n\times\PP^n:h(Z,W)=0\bigr\},
\]
with the natural projections $p_1$ and $p_2$. Since $h$ is
nondegenerate, $\mathcal{I}$ is a smooth irreducible hypersurface of
bidegree $(1,1)$.

Let
\begin{equation}\label{liftp}
P=(P_0,\ldots,P_N):\CC^{n+1}\longrightarrow\CC^{N+1}
\end{equation}
be a reduced homogeneous lift of $F$, of common degree $m\geq1$,
chosen so that
\begin{equation}\label{P1z}
 P(1,z)=(q(z),p(z)).   
\end{equation}

Here \emph{reduced} means that the components $P_j$ have no common
nonconstant polynomial factor.

\begin{lemma}\label{lem:sphere}
There exists a unique bihomogeneous polynomial $Q(Z,W)$ of bidegree
$(m-1,m-1)$ such that
\begin{equation}\label{eq:sphere}
B\bigl(P(Z),P^\#(W)\bigr)=h(Z,W)Q(Z,W).
\end{equation}
Moreover, $Q^\#=Q$.
\end{lemma}

\begin{proof}
Though this result is well known, we include a proof for
completeness. Set
\[
R(Z,W):=B\bigl(P(Z),P^\#(W)\bigr).
\]

Since $F$ is proper and holomorphically extended across the
boundary, it follows that
\begin{equation}\begin{split}\label{7}
R(Z,\ov{Z})=0\ \text{whenever}\ H(Z)=0.
\end{split}\end{equation}
The affine hypersurface
\[
\mathcal X:=\{(Z,W)\in\CC^{n+1}\times\CC^{n+1}:h(Z,W)=0\}
\]
is irreducible. The real null diagonal $ \{(Z,\overline Z):H(Z)=0,\
Z\neq0\} $
 lies in the smooth locus of
$\mathcal X$ and is maximally totally real there. Hence (\ref{7})
implies $R(Z,W)=0$ on $\mathcal X$.
 Since $h$ is irreducible, it follows from Hilbert's Nullstellensatz that $h$ divides $R$ in the polynomial ring.
Comparing bidegrees yields \eqref{eq:sphere}, and the quotient $Q$
is unique.

Finally, $R^\#=R$ and $h^\#=h$. Applying $\#$ to \eqref{eq:sphere}
gives $hQ^\#=hQ$, hence $Q^\#=Q$.
\end{proof}

\begin{lemma}\label{lem:positivity}
For every nonzero $Z\in\CC^{n+1}$ satisfying $H(Z)=0$,
\[
Q(Z,\overline Z)>0.
\]
In particular, $h\nmid Q$.
\end{lemma}

\begin{proof}
Set
\[
r(z):=\|z\|^2-1,\qquad \rho(w):=\|w\|^2-1.
\]
Because $F$ extends holomorphically across $\partial\B^n$ and maps
$\partial\B^n$ into $\partial\B^N$, real-analytic division gives
\begin{equation}\label{rhof}
\rho(F(z))=r(z)\lambda(z)
\end{equation}
near $\partial\B^n$, where $\lambda$ is real-valued and
real-analytic. The function $u:=\rho\circ F$ is subharmonic,
negative in $\B^n$, and zero on $\partial\B^n$. The Hopf
 lemma gives $\partial_\nu u>0$ on $\partial\B^n$,
where $\nu$ is the outward unit normal. Consequently,
\begin{equation}\label{lamda}
\lambda=\frac{\partial_\nu u}{\partial_\nu r}>0 \qquad\text{on
}\partial\B^n.
\end{equation}

By \eqref{P1z}, \eqref{eq:sphere} and \eqref{rhof}, we
obtain
\[
\begin{aligned}
r(z)Q\bigl((1,z),(1,\overline z)\bigr) &=\|p(z)\|^2-|q(z)|^2
=|q(z)|^2\rho(F(z))= |q(z)|^2 r(z)\lambda(z).
\end{aligned}
\]
Canceling $r(z)$ away from the sphere and then using continuity
gives
\[
Q\bigl((1,z),(1,\overline z)\bigr)=|q(z)|^2\lambda(z)>0 \qquad
\text{for}\ z\in\partial\B^n.
\]
The positivity follows from Lemma~\ref{lem:nopole} and \eqref{lamda}.

If $Z\neq0$ and $H(Z)=0$, then $Z_0\neq0$ and
$z=(Z_1/Z_0,\ldots,Z_n/Z_0)\in\partial\B^n$. Bihomogeneity now
yields
\[
Q(Z,\overline Z) =|Z_0|^{2m-2}Q\bigl((1,z),(1,\overline z)\bigr)>0.
\]
If $h$ divided $Q$, the left-hand side would vanish on the null
cone, contradicting this strict positivity.
\end{proof}

\subsection{From the CR Gauss map to the algebraic Gauss map}\label{subsec:boundary-projective}

This section proves the boundary-to-projective passage used in the
rest of the paper.
Let $F$ be defined in the previous section. We first show that $F$
is an immersion.
\begin{lemma}\label{rank}
Under the above hypothesis on $F$, $DF_z$ has complex rank $n$ at every
$z\in\partial\B^n$.
In particular, $f$ is a CR immersion.
\end{lemma}
\begin{proof}
Applying $\partial$ and $\p\bar\partial$ to (\ref{rhof}), respectively, we obtain at each point $z\in\partial\mathbb B^n$
\begin{equation*}\begin{split} 
\p \rho \circ \p F (v)=\lambda (z) \p r (v)\quad &\text{for } \ v\in T_z^{1,0}\CC^n,\\ \|\p F_z(\hat{v})\|^2=\lambda(z)\|\hat{v}\|^2 \quad &\text{for }\ \hat{v}\in T_z^{1,0}\p\BB^n. \end{split}\end{equation*}
If $\partial F_z(v)=0$, then the first identity implies that $\partial r_z(v)=0$, and hence
$v\in T_z^{1,0}\partial\mathbb B^n$. The second identity then yields
 $v=0$,  and the lemma follows immediately.
\end{proof}

\subsubsection{CR Gauss map}

We next recall the boundary Gauss map in the sphere model.

For $1\leq j<k\leq n$, define
\[
L_{j,k}:=\overline z_k\frac{\partial}{\partial z_j}
         -\overline z_j\frac{\partial}{\partial z_k},
\qquad T:=\operatorname{Re}\Big( i\sum_{j=1}^n
z_j\frac{\partial}{\partial z_j}\Big).
\]
At each $z\in\partial\B^n$, the vectors $\{L_{j,k}(z):\ 1\leq
j<k\leq n\}$ span $T_z^{1,0}\partial\B^n$, and the complexified real
tangent space satisfies
\[
\CC T_z\partial\B^n =\Span_{\CC}\bigl\{
L_{j,k}(z),\overline{L_{j,k}}(z),T(z):1\leq j<k\leq n \bigr\}.
\]
Since $\overline{L_{j,k}}f=0$, we define the CR Gauss map by
\begin{align}\label{eq:CRGauss}
\gamma_f(z):= \Span_{\CC}\big\{ L_{j,k}f(z),Tf(z):1\leq j<k\leq n
\big\}\in\text{Gr}(n,\CC^N),
\end{align}
where $\text{Gr}(n,\CC^N)$ denotes the Grassmannian of complex
$n$-dimensional linear subspaces of $\CC^N$. Equivalently,
\[
\gamma_f(z)=\im DF_z.
\]
We call $\gamma_f$ degenerate if its generic real fibers have
positive dimension, or equivalently if its generic real differential
rank is less than $2n-1$.

For comparison with the Heisenberg formulation used in
Cheng--Ji--Liu \cite{CJL13} and Yin--Yuan--Zhang \cite{YYZ20}, set
\[
\HH^n:=\{(z,w)\in\C^{n-1}\times\C:\operatorname{Im}w>\|z\|^2\}.
\]
The Cayley transformation
\begin{equation}\label{eq:cayley}
\rho_n:\HH^n\longrightarrow\BB^n, \qquad \rho_n(z,w)=
\left(\frac{2z}{1-iw},\frac{1+iw}{1-iw}\right)
\end{equation}
is biholomorphic and extends holomorphically across every finite
point of $\partial\HH^n$. Its boundary restriction identifies
$\partial\HH^n$ with $\partial\BB^n\setminus\{(0,\ldots,0,-1)\}$.
Wherever the composition is defined, the corresponding Heisenberg
representative is
\[
\widetilde F=\rho_N^{-1}\circ F\circ\rho_n.
\]
The degeneracy of the CR Gauss map defined on the sphere is
equivalent to that defined on the Heisenberg hyperplane. Our
arguments are carried out only in the sphere model, using the
homogeneous lift $P$ and the polarized identity \eqref{eq:sphere}.

\subsubsection{Projective Gauss map}
For the homogeneous lift $P$ defined by
\eqref{liftp}, on the locus where $[P]$ is defined and immersive, define the associated projective Gauss map
\[
\Gamma_{\rm proj}([Z])
:=
\Span\{P(Z),DP_Z(\CC^{n+1})\}
\in \operatorname{Gr}(n+1,\CC^{N+1}).
\]
Homogeneity of $P$ makes this subspace independent of the choice of
representative $Z$ of $[Z]$, while Euler's identity
$DP_Z(Z)=mP(Z)$ identifies the radial direction with the line spanned
by $P(Z)$. Let
\begin{align}\label{eq:projvarP}
X:=\overline{[P](\PP^n\setminus\operatorname{Bs}[P])}^{\,\mathrm{Zar}}\subset\PP^N
\end{align}
where the image is taken over the locus on which the rational map
$[P]:\PP^n\dashrightarrow\PP^N$ is defined. If $[P]$ is immersive at
$[Z]$ and $x=[P]([Z])\in X_{\mathrm{reg}}$, then by definition, we have
\[
\Gamma_{\rm proj}([Z])
=
\Span\{P(Z),DP_Z(\CC^{n+1})\}
=
\widehat T_xX.
\]
Recalling that $\Gamma_X$
denotes the projective Gauss map of $X$ introduced in \eqref{eq:projGauss}, we obtain
\begin{equation}\label{eq:Gamma-factor}
\Gamma_{\rm proj}=\Gamma_X\circ[P]
\end{equation}
wherever both sides are defined and $[P]$ is immersive.

\begin{proposition}\label{prop:bridge}
Under the assumption in Theorem~\ref{thm:main}, $[P]$ has generic projective rank $n$,
so that the projective
variety $X$ defined in \eqref{eq:projvarP} has dimension $n$. The following are equivalent:
\begin{enumerate}
\item[(i)] The CR Gauss map $\gamma_f$ is generically degenerate on
$\partial\mathbb B^n$.

\item[(ii)] The projective Gauss map $\Gamma_{\rm proj}$ is
generically degenerate; equivalently,
\begin{align}\label{eq:rank<n}
\operatorname{rank}_{\CC}d\Gamma_{\rm proj}<n
\end{align}
generically on $\PP^n$.

\item[(iii)] The projective Gauss map $\Gamma_X$ of $X$ is generically degenerate.
\end{enumerate}
\end{proposition}

\begin{proof}
On the affine chart $Z=(1,z)$, we have
$[P(1,z)]=[1:F(z)]$. Since $F$ has rank $n$ on a nonempty boundary
patch, $[P]$ has generic projective rank $n$, and hence $\dim X=n$.
Moreover, on the dense open set where $[P]$ is immersive and its
image lies in $X_{\rm reg}$, \eqref{eq:Gamma-factor} gives
\[
\Gamma_{\rm proj}=\Gamma_X\circ[P].
\]
There $d[P]$ is an isomorphism onto the tangent space of $X$, so
$d\Gamma_{\rm proj}$ and $d\Gamma_X$ have the same rank. Thus
\emph{(ii)} and \emph{(iii)} are equivalent.

Assume \emph{(i)}. Choose a nonempty relatively open boundary patch
contained in the maximal-rank stratum of $\gamma_f$ on which its
fibers have positive dimension. For every $z$ in this patch there is
$0\ne v\in T_z(\partial\mathbb B^n)$ such that
$d\gamma_{f,z}(v)=0$. Set $E_z=dF_z(\CC^n)$. Since
\[
T_{E_z}\operatorname{Gr}(n,\CC^N)
\simeq \operatorname{Hom}(E_z,\CC^N/E_z),
\]
differentiating $dF_z(\xi)$ in the direction $v$ gives
\[
d\gamma_{f,z}(v)(dF_z\xi)
=[D^2F_z(v,\xi)]\in\CC^N/E_z,
\]
and therefore
\begin{equation}\label{eq:boundary-second}
D^2F_z(v,\xi)\in E_z
\qquad\text{for every }\xi\in\CC^n.
\end{equation}
On the same affine chart, Lemma~\ref{lem:nopole} gives $q(z)\ne0$ and
$P(1,z)=q(z)(1,F(z))$, hence
\begin{equation}\label{eq:affine-projective-plane}
\Gamma_{\rm proj}([1:z])
=\Span\bigl\{(1,F(z)),(0,dF_z(\CC^n))\bigr\}.
\end{equation}
Differentiating these generators in the direction $v$, we get that their derivatives lie in $\Gamma_{\rm proj}([1:z])$ by
\eqref{eq:boundary-second}; hence $d\Gamma_{\rm proj,[1:z]}(v)=0$.
Thus \eqref{eq:rank<n} holds on this boundary patch. In local
Grassmannian coordinates, every $n\times n$ minor of
$d\Gamma_{\rm proj}$ is holomorphic near the patch and vanishes on a
relatively open subset of the real hypersurface $\partial\mathbb B^n$,
so it vanishes identically in a complex neighborhood. Since
$\Gamma_{\rm proj}$ is rational on the irreducible variety $\PP^n$,
these minors vanish identically as rational functions. Hence
\eqref{eq:rank<n} holds generically, proving \emph{(ii)}.

Conversely, assume \emph{(ii)}. On a nonempty boundary patch where
$F$ has rank $n$ and $\Gamma_{\rm proj}$ is regular,
$\ker_{\CC}d\Gamma_{\rm proj}$ has positive complex dimension. Since
$T_z(\partial\mathbb B^n)$ has real codimension one in $\CC^n$, there
is a nonzero
\[
v\in \ker_{\CC}d\Gamma_{\rm proj,[1:z]}
\cap T_z(\partial\mathbb B^n).
\]
Using \eqref{eq:affine-projective-plane}, we see that the condition $d\Gamma_{\rm proj,[1:z]}(v)=0$ implies
$D^2F_z(v,\xi)\in E_z$ for every $\xi\in\CC^n$. By the
Grassmannian differential formula above, $d\gamma_{f,z}(v)=0$.
Therefore $\gamma_f$ has real differential rank $<2n-1$ on a
nonempty boundary patch, and hence is generically degenerate. This
proves \emph{(i)} and completes the equivalence.
\end{proof}

\section{Gauss-kernel geometry, divisor separation, and branch dichotomy}\label{sec:kernel-divisor}

By Proposition~\ref{prop:bridge}, after shrinking to a nonempty
Zariski open set $U\subset\PP^n$ on which $[P]$ is immersive and
$\Gamma_{\rm proj}$ is regular of constant rank, the bundle
\[
\cK:=\ker d\Gamma_{\rm proj}\subset TU
\]
has positive rank $r:=\rk\cK$.  Choose a rational line field
$\kappa\subset\cK$ and a rational cone representative $v(Z)$ modulo
the radial line $\CC Z$.

\subsection{Gauss-kernel geometry and the canonical operator}\label{subsec:canonical-operator}

We first establish the generic injectivity of the cone differential
and use it to associate a canonical rational endomorphism to the
chosen Gauss-kernel direction.

Let $\kcone=\CC(Z_0,\dots,Z_n)$ be the function field of the source
cone, as fixed above.

\begin{lemma}\label{lem:DPinjective}
The $\kcone$-linear map
\[
DP:\kcone^{n+1}\longrightarrow\kcone^{N+1}
\]
is injective, and $P\in\im DP$.
\end{lemma}

\begin{proof}
Generic projective immersion implies that the induced map
\[
d[P]:\kcone^{n+1}/\kcone Z\longrightarrow\kcone^{N+1}/\kcone P
\]
is injective. If $DP(\xi)=0$, then $d[P]([\xi])=0$. By injectivity
of $d[P]$, we have $[\xi]=0$, and hence $\xi=cZ$ for some
$c\in\kcone$.  Euler's identity gives
\[
DP_Z(Z)=m\,P(Z),
\]
hence $0=DP_Z(cZ)=cm\,P(Z)$. Since $P$ is nonzero, it follows that
$c=0$. Thus $DP$ is injective.  Euler also gives
$P(Z)=m^{-1}DP_Z(Z)\in\im DP$.
\end{proof}

\begin{lemma}\label{lem:canonicalA}
Fix the rational cone representative $v$ chosen above. There is a
unique
\[
A\in\operatorname{End}_{\kcone}(\kcone^{n+1})
\]
such that
\begin{equation}\label{eq:canonicalA}
D^2P(v,\xi)=DP(A\xi)
\end{equation}
for every constant vector $\xi\in\mathbb C^{n+1}$.  The coefficients
of $A$ are rational functions of $Z$.
\end{lemma}

\begin{proof}
Since $v$ represents a direction in $\ker d\Gamma_{\rm proj}$, the
Grassmannian differential of $\Gamma_{\rm proj}$ vanishes in the direction
$v$. Differentiating the generator $DP(\xi)$, with $\xi$ constant,
therefore gives
\[
D^2P(v,\xi)\in\Span\{P,DP(\kcone^{n+1})\}.
\]
By Lemma~\ref{lem:DPinjective}, $P\in\im DP$, so the right-hand side
is simply $\im DP$. By injectivity of $DP$, for each constant $\xi$
there is a unique $A\xi\in \kcone^{n+1}$ satisfying
$D^2P(v,\xi)=DP(A\xi)$. The resulting assignment $\xi\mapsto A\xi$
is $\mathbb C$-linear and hence extends uniquely to a
$\kcone$-linear endomorphism
$A\in\operatorname{End}_{\kcone}(\kcone^{n+1})$. A rational left
inverse obtained from any nonzero $(n+1)\times(n+1)$ minor of $DP$
shows that the entries of $A$ are rational.
\end{proof}

We now combine the canonical endomorphism $A$ with the polarized
sphere identity to obtain the differential operator that will be
used throughout the incidence argument.

For a constant vector $\xi\in\mathbb C^{n+1}$, put
\[
\ell_v:=h(v,W),\qquad \ell_\xi:=h(\xi,W),\qquad \mu_\xi:=h(A\xi,W).
\]
For any $G\in\kcone[W_0,\dots,W_n]$, define the \emph{canonical operator} $\Dop$ by
\begin{equation}\label{eq:operator}
\Dop_\xi G := \mu_\xi G -\ell_vD_\xi G -\ell_\xi D_vG -hD^2_{\xi,v}G
+hD_{A\xi}G.
\end{equation}
Here all directional derivatives involving $v$ or $A\xi$ are
evaluated with the corresponding vector frozen at the base point; in
particular, $D^2_{\xi,v}$ denotes the Hessian evaluated on
$(v,\xi)$.

\begin{lemma}\label{lem:exactoperator}
For every constant $\xi\in \mathbb C^{n+1}$,
\begin{equation}\label{eq:DQ0}
\Dop_\xi Q=0
\end{equation}
as an identity in $\kcone[W_0,\dots,W_n]$.
\end{lemma}

\begin{proof}
Differentiating \eqref{eq:sphere} in the  vectors $v$, $D_\xi v$ and
$A\xi$, respectively, yields
\begin{align}
&B(DP(v),P^\#)=\ell_vQ+hD_vQ,\ \ \label{eq:firstv} \\
&B(DP(D_\xi
v),P^\#)=h(D_\xi v,W)Q+hD_{D_\xi v}Q, \label{2}     \\
&B(DP(A\xi),P^\#)=\mu_\xi Q+hD_{A\xi}Q.  \label{3}
\end{align}
Since $v$ varies with $Z$, a direct computation shows that
\[
D_\xi(DP(v))=D^2P(v,\xi)+DP(D_\xi v),\ \
D_\xi(D_vQ)=D^2_{\xi,v}Q+D_{D_\xi v}Q.
\]
Hence by differentiating \eqref{eq:firstv} in the  direction $\xi$,
we obtain
\begin{equation}\begin{split}\label{eq:secondidentity}
&B(D^2P(v,\xi),P^\#)+ B(DP(D_\xi v),P^\#)\\
=&h(D_\xi v,W)Q+\ell_vD_\xi Q+\ell_\xi D_vQ+hD_{D_\xi
v}Q+hD^2_{\xi,v}Q.
\end{split}\end{equation}
Together with (\ref{2}), we get
\begin{equation}\label{eq:2}
B(D^2P(v,\xi),P^\#) =\ell_vD_\xi Q+\ell_\xi D_vQ+hD^2_{\xi,v}Q.
\end{equation}
Combining (\ref{3}) with Lemma~\ref{lem:canonicalA} yields
\begin{equation}\begin{split}\label{4}
B(D^2P(v,\xi),P^\#)=\mu_\xi Q+hD_{A\xi}Q.
\end{split}\end{equation}
Now $\Dop_\xi Q=0$ follows directly from (\ref{eq:operator}),
(\ref{eq:2}) and (\ref{4}).
\end{proof}

\subsection{Residual logarithmic equation and divisor separation}

We next extract the characteristic divisor associated with $\kappa$
and use the exact operator identity to control the remaining divisor
of $Q|_{\cI}$.

On the domain of the rational line field $\kappa$, the condition
\[
h(v(Z),W)=0
\]
defines a hyperplane in each incidence fiber $p_1^{-1}([Z])$.
Denote by $\Delta_\kappa$ the Zariski closure in $\cI$ of
\[
\Delta_\kappa^\circ := \left\{ ([Z],[W])\in p_1^{-1}(U): h(v(Z),W)=0
\right\}.
\]
Here $U$ is a Zariski open subset of $\mathbb{P}^n$ on which $[P]$ is immersive and
$\Gamma_{\rm proj}$ is regular of constant rank.

\begin{lemma}\label{lem:intrinsic}
The indeterminacy locus of $\kappa$ has codimension at least two.
The closure $\Delta_\kappa$ is an irreducible divisor, independent
of the chosen cone representative of $\kappa$, which is defined as the characteristic divisor of $\kappa$.
\end{lemma}

\begin{proof}
The projective bundle $\PP(T\PP^n)\to\PP^n$ is proper.  A rational
section over the normal variety $\PP^n$ extends across every
codimension-one point by the valuative criterion for properness
(cf.~\cite[II.4]{Ha77}), so the indeterminacy locus has codimension
at least two.  Over the domain $U$ of $\kappa$, the annihilator
condition defines a relative hyperplane in the incidence fibers.
Thus it is a $\mathbb P^{n-2}$-bundle over the irreducible variety
$U$, and is therefore irreducible.  Its closure $\Delta_\kappa$ in
$\mathcal I$ is therefore irreducible and has codimension one, hence
is an irreducible divisor.

If $v'$ is another rational cone representative, then $v'=fv+gZ$ for
some rational functions $f,g$, with $f\not\equiv0$. On $\cI$,
\[
h(v',W)=f(Z)h(v,W).
\]
Hence the two relative hyperplanes agree on the dense open set where
$f\ne0$, and therefore have the same closure.
\end{proof}

Put $s:=\ord_{\Delta_\kappa}(Q|_{\cI})\ge0.$ Near the generic point
of $\Delta_\kappa$, write
\begin{equation}\label{eq:localfactor}
Q|_{\cI}=\ell_v^sR, \qquad \ord_{\Delta_\kappa}R=0.
\end{equation}

We shall use the following linear-algebra lemma to analyze the
residual factor along $\Delta_\kappa$.
For $S\in\operatorname{End}(\mathbb C^{n+1})$, let $S^{\dagger}$
denote its transpose with respect to $h$, defined by
\[
h(Su,W)=h(u,S^{\dagger}W).
\]

\begin{lemma}\label{lem:eigenvector}
Let $Z,v$ be linearly independent and set
\[
K_{Z,v}:=\{W\in\CC^{n+1}:h(Z,W)=h(v,W)=0\}.
\]
Suppose that $S^{\dagger}W\in\CC W$ for a Zariski dense set of
nonzero $W\in K_{Z,v}$.  Then there exist $\lambda\in\mathbb C$ and
covectors $a,b\in(\mathbb C^{n+1})^*$ such that
\begin{equation}\label{eq:ranktwoform}
S=\lambda I+Z\otimes a+v\otimes b.
\end{equation}
\end{lemma}

\begin{proof}
The condition $W\wedge S^{\dagger}W=0$ is Zariski closed, hence
\begin{equation}\begin{split}\label{8}
S^\dagger W=\lambda(W)W \ \ \text{for every nonzero vector of $K_{Z,v}$}.
\end{split}\end{equation}
First assume  $\dim K_{Z,v}\ge2$. Applying (\ref{8}) to  $W_1$, $W_2$ and
$W_1+W_2$ for any linearly independent $W_1,W_2\in K_{Z,v}$ yields
$$
S^\dagger
(W_1+W_2)=\lambda(W_1+W_2)(W_1+W_2)=\lambda(W_1)W_1+\lambda(W_2)W_2.
$$
The independence of $W_1$ and $W_2$ gives
$\lambda(W_1+W_2)=\lambda(W_1)=\lambda(W_2)$. Hence  the eigenvalue
is independent of $W$. This is trivially true when $\dim K_{Z,v}=1$. Hence
$S^{\dagger}|_{K_{Z,v}}=\lambda I$ for some $\lambda\in\mathbb C$.

Therefore, for every $u\in\mathbb C^{n+1}$ and $W\in K_{Z,v}$,
\[
h((S-\lambda I)u,W)=h(u,(S^{\dagger}-\lambda I)W)=0.
\]
Hence
\[
\im(S-\lambda I)\subset K_{Z,v}^{\perp_h}=\Span\{Z,v\}.
\]
Since $Z$ and $v$ are linearly independent, there exist covectors
$a,b$ such that $S-\lambda I=Z\otimes a+v\otimes b$, which proves
\eqref{eq:ranktwoform}.
\end{proof}

Applying this lemma to the endomorphism arising from the exact
operator identity yields a logarithmic differential equation for the
residual factor.

\begin{proposition}\label{prop:log}
On a nonempty Zariski open subset of $\mathbb P^n$, there is a
rational covector $b_Z$ such that, whenever $\xi$ is tangent to the
$p_2$-fiber through $([Z],[W])$, equivalently
$h(\xi,W)=0$,
one has
\begin{equation}\label{eq:log}
D_\xi R=b_Z(\xi)R.
\end{equation}
\end{proposition}

\begin{proof}
The factorization \eqref{eq:localfactor} is an identity on $\cI$,
not a priori in the ambient product.  Choose local ambient
representatives.  After shrinking the open set, there is a
holomorphic function $S$ such that
\begin{equation}\label{eq:ambientresidual}
Q=\ell_v^sR+hS.
\end{equation}
Let $\xi$ be tangent to the $p_2$-fiber, so $h(\xi,W)=0$.  On $\cI$, $h(Z,W)=0$,
 \eqref{eq:DQ0} becomes 
\begin{equation}\begin{split}\label{q1}
\mu_\xi Q-\ell_vD_\xi Q=0.
\end{split}\end{equation}
 On the other hand, $D_\xi\ell_v=h(D_\xi v,W)$ and
\begin{equation}\begin{split}\label{q2}
&D_\xi(\ell_v^sR)=s\ell_v^{s-1}h(D_\xi v,W)R+\ell_v^{s}D_\xi R,\\
&D_\xi(hS)=h(\xi,W)S+hD_\xi S=0.
\end{split}\end{equation}
Substituting \eqref{eq:ambientresidual} and (\ref{q2}) into (\ref{q1}) gives
$$
\mu_\xi \ell_v^sR-\ell_v\big(s\ell_v^{s-1}h(D_\xi v,W)R+\ell_v^{s}D_\xi R\big)=0.
$$
Note that  $\mu_\xi=h(A\xi,W)$, and write $D_\xi
v=(Dv)\xi$. Then, on the dense open set where $\ell_v\ne0$,
\begin{equation}\label{eq:residualpre}
h((A-sDv)\xi,W)R=\ell_vD_\xi R.
\end{equation}
By rational continuation, \eqref{eq:residualpre} holds identically
on the chosen open set.

Write $\widetilde A:=A-sDv.$ By the definition of $s$,
$\ord_{\Delta_\kappa}R=0$. Hence $R$ is nonzero on a Zariski dense
open subset of $\Delta_\kappa$. For general $Z$ and general $W\in
K_{Z,v}$,  equation \eqref{eq:residualpre} therefore gives
\[
h(\widetilde A\xi,W)=0 \qquad\text{for every }\xi\text{ satisfying
}h(\xi,W)=0.
\]
Equivalently
\[
\widetilde A^{\dagger}W\in\CC W
\]
for a Zariski dense set of $W\in K_{Z,v}$.
Lemma~\ref{lem:eigenvector} gives
\[
\widetilde A=\lambda I+Z\otimes a_Z+v\otimes b_Z.
\]
Substituting this identity into \eqref{eq:residualpre} and using $h(Z,W)=h(\xi,W)=0$ yields
\[
b_Z(\xi)\ell_vR=\ell_vD_\xi R.
\]
Canceling $\ell_v$ on the dense open set where $\ell_v\ne0$ gives
\eqref{eq:log}. Since all the preceding data are defined over
$\kcone$, the covector $b_Z$ may be chosen rationally in $Z$.
\end{proof}

\begin{corollary}\label{cor:tangency}
Let $D$ be a prime component of the residual zero divisor of $R$
such that  $p_1(D)$ is dense in $\PP^n$.  Let $U\subset\PP^n$
be a nonempty Zariski open set on which the rational data $v,A,b_Z$
and the local trivializations used in Proposition~\ref{prop:log} are
regular.  Then, at a general smooth point
\[
x:=([Z],[W])\in D\cap p_1^{-1}(U)
\]
away from $\Delta_\kappa$ and the other residual components,
\begin{equation}\label{equ:TxD}
T_x\bigl(p_2^{-1}(p_2(x))\bigr)\subset T_xD.
\end{equation}
\end{corollary}

\begin{proof}
Since $p_1(D)$ is dense in $\PP^n$, $D\cap p_1^{-1}(U)$ is dense
in $D$. After removing $\Delta_\kappa$, the other residual
components, and the singular locus of $D$, the indicated general
point may therefore be chosen.  Near such a point the residual
factor $R$ and the covector $b_Z$ are regular.  Write
\[
R=g^k u, \qquad k\ge1,
\]
where $g=0$ is a reduced local equation of $D$ and $u$ is a local
unit.  For a $p_2$-fiber tangent vector $\xi$, \eqref{eq:log} gives
\[
kg^{k-1}uD_\xi g+g^kD_\xi u=b_Z(\xi)g^k u.
\]
Every coefficient in this identity is regular at the generic point
of $D$.  Dividing by $g^{k-1}$ and reducing modulo $g$ therefore
gives
\[
k u D_\xi g\equiv0\pmod g.
\]
Since the ground field has characteristic zero and $u$ is a unit, it
follows that $D_\xi g\equiv0\pmod g$. Thus $D_\xi g=0$ along $D$. In
particular, at the chosen smooth point $x$, $dg_x(\xi)=0$, so
$\xi\in T_xD$. Since $\xi$ was an arbitrary tangent vector to the
$p_2$-fiber, this proves \eqref{equ:TxD}.
\end{proof}


We now combine the preceding tangency result with a fiber-dimension
argument to show that every residual prime component is pulled back
through one of the two incidence projections.

\begin{proposition}[Global separation]\label{prop:global}
There are effective divisors $D_1,D_2$ on $\PP^n$ such that
\begin{equation}\label{eq:globaldiv}
\Div(Q|_{\cI})=s\Delta_\kappa+p_1^*D_1+p_2^*D_2.
\end{equation}
\end{proposition}

\begin{proof}
Let $D$ be a prime component of
\[
\Div(Q|_{\cI})-s\Delta_\kappa.
\]  
Since $\dim\cI=2n-1$, one has $\dim D=2n-2$. The proof is divided according to whether $p_1(D)$ is dense in $\PP^n$.

{\bf Case I:} Suppose  first that
$B_1:=p_1(D)$  is a proper irreducible subset
of $\PP^n$.  Every fiber of $p_1$ has dimension $n-1$, so
\[
\dim D\le \dim B_1+n-1.
\]
Thus $\dim B_1\ge n-1$.  Since $B_1$ is a proper irreducible subset
of $\PP^n$, necessarily $\dim B_1=n-1$, so $B_1$ is a prime divisor.
The irreducible variety $p_1^{-1}(B_1)$ has dimension $2n-2$ and
contains $D$, hence
\[
D=p_1^{-1}(B_1).
\]

{\bf Case II:}
Now suppose that $p_1(D)$ is dense in $\PP^n$.  Choose a nonempty
Zariski open set $U\subset\PP^n$ on which the rational data
$v,A,b_Z$ and the local trivializations of
Proposition~\ref{prop:log} are regular.  Since $p_1(D)$ is dense in $\PP^n$,
we may choose a general smooth point $x:=([Z],[W])\in D\cap
p_1^{-1}(U)$ away from $\Delta_\kappa$ and all other residual
components.  By Corollary~\ref{cor:tangency},
\[
T_x(p_2^{-1}(p_2(x)))\subset T_xD.
\]
The left side has dimension $n-1$, so
\[
\rk d(p_2|_D)_x\le (2n-2)-(n-1)=n-1.
\]
Since this holds at a general point of $D$, the generic rank of
$p_2|_D$ is at most $n-1$. Hence $\dim p_2(D)\le n-1$.  On the other
hand, every fiber of $p_2|_D$ is contained in the corresponding
fiber of $p_2:\cI\to\PP^n$, which has dimension $n-1$, hence the fiber-dimension theorem gives
\[
\dim p_2(D)\ge (2n-2)-(n-1)=n-1.
\]
Hence $\dim p_2(D)=n-1$, so $B_2:=p_2(D)$ is a prime divisor. The irreducible variety $p_2^{-1}(B_2)$ has dimension $2n-2$ and contains $D$, hence
\[
D=p_2^{-1}(B_2).
\]
Collecting all prime components with multiplicities proves
\eqref{eq:globaldiv}.
\end{proof}

\subsection{Geometric branch dichotomy}

We now distinguish the possible characteristic behavior according to
the generic rank $r=\rk\cK$ and the multiplicity
$s=\ord_{\Delta_\kappa}(Q|_{\cI})$.
We call $Q$ \emph{non-characteristic} with respect to $\kappa$ if
$s=0$, and \emph{characteristic} with respect to $\kappa$ if $s>0$.

\begin{proposition}\label{prop:rge2}
If the generic rank $r=\rk\cK$ is at least two, then there exists a
rational kernel line field with characteristic multiplicity $s=0$.
\end{proposition}

\begin{proof}
Work over the generic point of $\PP^n$.  On the generic incidence
hyperplane
\[
H_Z:=\{[W]\in\PP^{n}:h(Z,W)=0\},
\]
Lemma~\ref{lem:positivity} gives $h\nmid Q$, and hence
$Q(Z,W)|_{H_Z}$ is nonzero. It therefore has only finitely many
projective linear factors.  The map
\[
\PP(\cK_Z)\longrightarrow\PP(H_Z^*), \qquad [v]\longmapsto
[h(v,W)|_{H_Z}],
\]
is injective. Indeed, if two such restrictions are proportional,
then for some $c\ne0$, $h(v-cv',W)=0$ for all $W\in H_Z$. Hence
\[
v-cv'\in H_Z^{\perp_h}=\CC Z,
\]
so $v$ and $v'$ determine the same projective tangent direction.
Since $\dim\PP(\cK_Z)=r-1\ge1$, we may choose $[v]$ outside the
finite set of kernel directions for which $h(v,W)|_{H_Z}$ is a
linear factor of $Q|_{H_Z}$. Then $\ell_v|_{H_Z}\nmid Q|_{H_Z}$.
Since $\Delta_\kappa$ is generically defined in the incidence fiber
by $\ell_v|_{H_Z}=0$, it follows that
$s=\ord_{\Delta_\kappa}(Q|_{\cI})=0.$

The chosen line over the function field defines a rational line
field $\kappa\subset\cK$ on a nonempty Zariski open subset of
$\PP^n$.
\end{proof}

In the rest of this subsection, assume $r=1$ and $s\ge1$.
\begin{lemma}We have
\[
[\Delta_\kappa]=aH_1+H_2
\]
for some integer $a$, where $H_i:=p_i^*\cO_{\PP^n}(1)$ are the two
hyperplane classes on $\cI$.
\end{lemma}

\begin{proof}
 Since the incidence variety $\mathcal{I}$ is the projectivization over $\PP^n$ of the
 bundle
 $E=\{([Z],W)\in
  \PP^n\times\CC^{n+1}:\ h(Z,W)=0\}$, the projective-bundle formula (cf.~\cite[III.5]{Ha77}) gives
$$
\text{Pic}(\PP(E))=\text{Pic}(\PP^n)\oplus
\mathbb{Z}[\cO_{\PP(E)}(1)].
$$
Namely, $\text{Pic}(\mathcal{I})=\mathbb{Z}H_1+ \mathbb{Z}H_2$. Thus
$[\Delta_\kappa]=aH_1+bH_2$ for some $a,b\in \mathbb{Z}$. Since
$\Delta_\kappa$ cuts a general $p_1$-fiber into a hyperplane, we have $b=1$.
\end{proof}

\begin{lemma}\label{lem:a-positive}
One has $a\ge0$.
\end{lemma}

\begin{proof}
An effective divisor of class $aH_1+H_2$ gives a nonzero section of
$\cO_{\cI}(a,1)$.  From
\[
0\to\cO(a-1,0)\xrightarrow{\cdot h}\cO(a,1)\to\cO_{\cI}(a,1)\to0
\]
and the standard cohomology of line bundles on projective space ,
$H^0(\cO(a,1))=0$ and $H^1(\cO(a-1,0))=0$ when $a<0$ and $n\ge2$.
Hence $H^0(\cO_{\cI}(a,1))=0$, a contradiction.
\end{proof}

\begin{proposition}[Pure/mixed dichotomy]\label{prop:puremixed}
Exactly one of the following occurs.
\begin{enumerate}[label=(\alph*)]
\item \emph{Pure branch:} $a=0$ and
\[
\Delta_\kappa=p_2^{-1}(\Pi_L)
\]
for a hyperplane $\Pi_L:=\{L(W)=0\}\subset\PP^n$.
\item \emph{Mixed branch:} $a=1$, $\Delta_\kappa$ is the unique mixed prime component of $\Div(Q|_{\cI})$, and it is invariant under the Hermitian involution $([Z],[W])\mapsto([\overline W],[\overline Z])$.
\end{enumerate}
\end{proposition}

\begin{proof}
If $a=0$, the restriction sequence
\[
0\to\cO(-1,0)\xrightarrow{\cdot h}\cO(0,1)\to\cO_{\cI}(0,1)\to0
\]
shows that every section of $\cO_{\cI}(0,1)$ is the restriction of a
linear form $L(W)$. Since $\Delta_\kappa$ is an effective divisor of
class $H_2$, it follows that
\[
\Delta_\kappa = \{L(W)=0\}\cap\cI = p_2^{-1}(\Pi_L).
\]

If $a>0$,  then $[\Delta_\kappa]=aH_1+H_2$ has positive coefficients
in both $H_1$ and $H_2$, so $\Delta_\kappa$ is mixed. By
Proposition~\ref{prop:global}, every other prime component of
$\Div(Q|_{\cI})$ is a $p_1$- or $p_2$-pullback. Hence
$\Delta_\kappa$ is the unique mixed prime component. Since $Q^\#=Q$,
the divisor $\Div(Q|_{\cI})$ is invariant under the Hermitian
involution $\iota$, which interchanges $H_1$ and $H_2$. The
involution sends mixed prime components to mixed prime components.
Since $\Delta_\kappa$ is the unique mixed prime component, it
follows that $\iota(\Delta_\kappa)=\Delta_\kappa.$ Consequently,
\[
aH_1+H_2=[\Delta_\kappa] =[\iota(\Delta_\kappa)] =H_1+aH_2.
\]
Since $H_1$ and $H_2$ are independent in $\operatorname{Pic}(\cI)$,
it follows that $a=1$.
\end{proof}

The preceding results complete the geometric branch dichotomy. The
exclusion of the mixed branch and the resulting first-order
factorization are part of the classification argument and will be
taken up in Section~\ref{sec:classification}.

\section{Branch reduction and the complete classification}\label{sec:classification}
Section~\ref{sec:kernel-divisor} reduces the problem according to
$s=\ord_{\Delta_\kappa}(Q|_{\cI}).$ The case $s=0$ is
non-characteristic, while for $s>0$, Proposition~\ref{prop:puremixed}
divides the defect-one characteristic case into the pure and mixed
branches; Proposition~\ref{prop:rge2} shows that defect at least two
always yields a direction with $s=0$. In this section,
Proposition~\ref{prop:nomixed} excludes the mixed branch and
Proposition~\ref{prop:linear} shows that the case $s=0$ is totally
geodesic. The remaining pure characteristic case is treated in Proposition~\ref{prop:terminal} directly
by the geometry of its fixed center and the Griffiths--Harris theorem
on general fibers of a degenerate projective Gauss map.

\subsection{Exclusion of the mixed branch}

In this subsection, we assume the mixed branch defined in
Proposition~\ref{prop:puremixed} occurs, and try to reach a
contradiction. The restriction sequence for the $(1,1)$ incidence
hypersurface, together with the standard cohomology of projective
space, shows that every section of $\cO_{\cI}(1,1)$ lifts to a
bilinear form, uniquely modulo the one-dimensional space spanned by
$h$ (cf.~\cite[III.5]{Ha77}). Hence there is an endomorphism
$\Theta$ of $\CC^{n+1}$ such that
\begin{equation}\label{eq:kTheta}
\Delta_\kappa=\bigl\{([Z],[W])\in\cI:\ \ h(\Theta Z,W)|_{\cI}=0\bigr\}.
\end{equation}
Fiberwise equality of the characteristic hyperplanes lets us choose
the cone representative so that
\begin{equation}\label{eq:vThetaZ}
v(Z)=\Theta Z\pmod{\CC Z}.
\end{equation}
Adding a scalar multiple of the identity to $\Theta$ does not change
$h(\Theta Z,W)|_{\cI}$. Denote by $\Theta^\dagger$ the $h$-adjoint
of $\Theta$.

\begin{lemma}\label{lem:selfadjoint}
The representative $\Theta$ can be chosen such that
$\Theta=\ov{\Theta}^\dagger$.
\end{lemma}

\begin{proof}
By Lemma~\ref{lem:sphere}, $Q^\#=Q$, and by
Propositions~\ref{prop:global} and~\ref{prop:puremixed},
$\Delta_\kappa$ is the unique mixed prime component of
$\Div(Q|_{\cI})$.  Hence the $\#$-conjugate of $h(\Theta
Z,W)|_{\cI}$ defines the same divisor, and therefore
\[
\bigl(h(\Theta Z,W)|_{\cI}\bigr)^\# = c\,h(\Theta Z,W)|_{\cI}
\]
for some $c\in\CC^*$.  Applying $\#$ again gives
\[
\,h(\Theta Z,W)|_{\cI}=\big(\bigl(h(\Theta
Z,W)|_{\cI}\bigr)^\#\big)^\# = \bigl( c h(\Theta
Z,W)|_{\cI}\bigr)^\#=|c|^2\,h(\Theta Z,W)|_{\cI}.
\]
Hence $|c|=1$. After multiplying $\Theta$ by a suitable constant, we
may assume that $h(\Theta Z,W)|_{\cI}$ is $\#$-invariant. This
implies
\[
h(\Theta Z,W)=\bigl(h(\Theta Z,W)\bigr)^\#=h(\ov{\Theta}
W,Z)=h(Z,\ov{\Theta} W)=h(\ov{\Theta}^\dagger Z, W) \qquad\text{on
}\cI.
\]
Hence
\[
h((\ov{\Theta}^\dagger-\Theta)Z,W)=0 \qquad\text{on }\cI.
\]
Since this is a bilinear form of bidegree $(1,1)$ vanishing on
$\{h=0\}$, there is $\mu\in\CC$ such that
\[
{\ov{\Theta}}^\dagger-\Theta=\mu I.
\]
Taking adjoints gives $\overline\mu=-\mu$, so $\mu$ is purely
imaginary. Replacing $\Theta$ by $\Theta+\frac{\mu}{2}I$ therefore
makes it self-adjoint.  The above modification does not change
$h(\Theta Z,W)|_{\cI}$. And thus equation \eqref{eq:kTheta} remains
unchanged.
\end{proof}

We also need the following linearity theorem of general Gauss fibers
achieved by Griffiths--Harris~\cite[(2.10), p.~388]{GH79}.
\begin{theorem}[Griffiths--Harris]\label{thm:GH}
Let $X\subset\PP^N$ be a projective variety of dimension $n$,
and let $\Gamma_X$ be its projective Gauss map defined in \eqref{eq:projGauss}. Then there exists a dense open subset
$U\subset X_{\mathrm{reg}}$ such that for every tangent plane $T$, its fiber of the projective Gauss map
\[\Gamma_X^{-1}(T)=\{x\in X_{\mathrm{reg}}:\widehat{T}_xX=T\} \]
meeting $U$ is an open subset of a projective linear subspace.
Equivalently, the closure of a general fiber of $\Gamma_X$ is a
projective linear subspace.
\end{theorem}

\begin{lemma}\label{lem:kernelfiber}
Work on a Zariski open set on which $[P]$ is immersive, the Gauss
defect is one, and the mixed-branch kernel line is represented by
$v(Z)=\Theta Z$.  If $[Z(t)]$ is a local integral curve of this line
field, then $\Gamma_{\rm proj}([Z(t)])$ is constant.  For a general
initial point, $[P(Z(t))]$ is therefore contained in one general
fiber of the projective Gauss map of $X$, whose closure is a
projective line.
\end{lemma}

\begin{proof}
Along an integral curve,
\[
\frac d{dt}\Gamma_{\rm proj}([Z(t)])
=d\Gamma_{\rm proj}\bigl([\Theta Z(t)]\bigr)=0,
\]
so the tangent plane is constant.  On the chosen open set,
$\Gamma_{\rm proj}([Z])$ is precisely the embedded projective tangent
space of $X$ at $[P(Z)]$.  Hence the image curve lies in one fiber of
the projective Gauss map of $X$.  The defect is one, so the general
smooth fiber has dimension one; Theorem~\ref{thm:GH} says that its
projective closure is a line.  Locally the integral curve is
represented by $Z(t)=e^{t\Theta}Z_0$, and $P(Z(t))$ lies, for all
small $t$, in the fixed two-dimensional vector space over that
projective line. Moreover $P$ and $\dot P=DP(\Theta Z)$ are linearly
independent at a general point: otherwise the projective differential
of $[P]$ would kill the nonzero tangent class of $\Theta Z$,
contradicting immersivity.
\end{proof}
Lemma~\ref{lem:kernelfiber} provides the key geometric input for
proving that the minimal polynomial of $\Theta$ has degree at most
two.
\begin{lemma}\label{lem:minpoly}
In the mixed branch, the minimal polynomial of $\Theta$ has degree
at most two.
\end{lemma}

\begin{proof}
Choose the representative $v=\Theta Z$.  Since the first
differential $Dv=\Theta$, Proposition~\ref{prop:log} and its proof
give
\[
A-s\Theta=\lambda I+Z\otimes a+\Theta Z\otimes b,
\]
where $A$ is defined in \eqref{lem:canonicalA}. Applying this to
$\Theta Z$ yields
\begin{equation}\label{eq:minpoly1}
A(\Theta Z)-s\Theta^2Z\in\Span\{Z,\Theta Z\}.
\end{equation}

By Lemma~\ref{lem:kernelfiber}, for a general kernel integral curve
the vector $P(Z(t))$ remains in a fixed two-dimensional space, with
$P,\dot P$ forming a basis.  Hence
\[
\ddot P=D^2P(\Theta Z,\Theta Z)+DP(\Theta^2Z) \in\Span\{P,DP(\Theta
Z)\}.
\]
Using $D^2P(\Theta Z,\Theta Z)=DP(A(\Theta Z))$ and Euler's identity
$P=m^{-1}DP(Z)$, injectivity of $DP$ gives
\begin{equation}\label{eq:minpoly2}
A(\Theta Z)+\Theta^2Z\in\Span\{Z,\Theta Z\}.
\end{equation}
Subtracting \eqref{eq:minpoly1} from \eqref{eq:minpoly2} gives $
(s+1)\Theta^2Z\in\Span\{Z,\Theta Z\}.$

Let $\mathcal{C}_Z$ be the cyclic subspace generated by $Z$:
$$
\mathcal{C}_Z=\operatorname{span}_\CC\{Z,\Theta Z,\Theta^2 Z,\cdots\}
$$
A standard fact in finite-dimensional linear algebra is that the degree of minimal
polynomial of $\Theta$ is $\max\{\text{dim}({\mathcal{C}}_Z):\ Z\in \mathbb{C}^{n+1}\}$, which completes the proof of Lemma \ref{lem:minpoly}.
\end{proof}

\begin{lemma}\label{lem:indexoneTheta} 
In the mixed branch, there exist distinct
real numbers $\lambda,\mu$ and a negative line $E_-\subset\CC^{n+1}$
such that
\begin{equation}\label{eq:Thetaform}
\Theta=\mu I+(\lambda-\mu)\pi_{E_-},
\end{equation}
where $\pi_{E_-}$ is the $H$-orthogonal projection onto $E_-$.
\end{lemma}

\begin{proof}
By Lemmas~\ref{lem:selfadjoint} and~\ref{lem:minpoly},
$\Theta={\ov{\Theta}}^\dagger$ and its minimal polynomial has degree at
most two.  Moreover, if $0\ne Z$ is null and $h(\Theta Z,\ov Z)=0$, then
$([Z],[\overline Z])\in\Delta_\kappa$, so $Q(Z,\overline Z)=0$,
contradicting Lemma~\ref{lem:positivity}. Thus
\[
h(\Theta Z,\ov Z)\ne0
\]
for every nonzero null $Z$.

Hence $\Theta$ is not scalar and has no nonreal eigenvalue.  Its
minimal polynomial also cannot be $(x-\lambda)^2$: writing
$\Theta=\lambda I+\mathsf N$ with $\mathsf N^2=0$, any $e=\mathsf N
x\ne0$ is null and satisfies $h(\Theta e,\ov e)=0$. Therefore the
minimal polynomial has two distinct real roots.

The corresponding eigenspaces are nondegenerate and orthogonal.
Since $H$ has negative index one, exactly one has negative index
one; it must be one-dimensional, since otherwise it would contain a
nonzero null eigenvector.  Denoting this negative line by $E_-$
gives \eqref{eq:Thetaform}.
\end{proof}

\begin{proposition}\label{prop:nomixed}
The mixed branch (b) of Proposition~\ref{prop:puremixed} is
impossible.
\end{proposition}

\begin{proof}
By Lemmas~\ref{lem:selfadjoint}--\ref{lem:indexoneTheta}, $\Theta$
is a nonzero scalar multiple of the $H$-orthogonal projection onto a
negative line $E_-=\CC e$ up to $\mu I$.  Modulo the incidence
equation $h=0$, the defining form therefore separates, up to a
nonzero scalar, as
\[
h(\Theta Z,W)=h(Z,\ov e)\,h(e,W).
\]
Its divisor on $\cI$ is the sum of a horizontal and a vertical
divisor.  This contradicts the irreducibility and mixed divisor
class of $\Delta_\kappa$ proved in Lemma \ref{lem:intrinsic} and
Proposition \ref{prop:puremixed}.
\end{proof}

\subsection{The non-characteristic branch}

Assume there is a non-characteristic kernel direction, so that $s=0$ in Proposition~\ref{prop:global}. We will prove $F$ is linear and hence totally geodesic.

\begin{lemma}\label{lem:product}
For every $j\ge1$ and every $R\in\kcone[W]$,
\begin{align}\label{eq:product}
\Dop_\xi(h^jR)
=&-j(j+1)h^{j-1}\ell_v\ell_\xi R\\
&+h^j\bigl((j+1)(\mu_\xi R-\ell_vD_\xi R-\ell_\xi D_vR)-jh(D_\xi v,W)R\bigr)\notag\\
&+h^{j+1}\bigl(D_{A\xi}R-D^2_{\xi,v}R-D_{D_\xi v}R\bigr).\notag
\end{align}
\end{lemma}

\begin{proof}
Note that
\[
D_\xi h=\ell_\xi, \qquad D_vh=\ell_v, \qquad D_{A\xi}h=\mu_\xi,
\qquad D^2_{\xi,v}h=0.
\]
Hence a direct computation yields
\begin{equation*}\begin{split}
D_\xi(h^jR)=&jh^{j-1}\ell_\xi R+h^jD_\xi R,\\
D_v(h^jR)=&jh^{j-1}\ell_v
R+h^jD_vR,\\
 D_{A\xi}(h^jR)=&jh^{j-1}\mu_\xi
R+h^jD_{A\xi}R,\\
 D^2_{\xi,v}(h^jR)=&j(j-1)h^{j-2}\ell_\xi\ell_vR
+ jh^{j-1}h(D_\xi v,W)R\\
&+jh^{j-1}(\ell_vD_\xi R 
+\ell_\xi D_vR)
+h^jD^2_{\xi,v}R
+h^jD_{D_\xi v}R.
\end{split}\end{equation*}
Substituting these identities into \eqref{eq:operator} gives
\eqref{eq:product}.
\end{proof}

The following coefficient-space fact will be used at the end of the non-characteristic argument to transfer a common factor of the \(Z\)-coefficients of \(B(P,P^\#)\) to the components of the reduced lift \(P\).

Let $\md(Z)$ be a monomial basis of $\CC[Z]_m$, then there exists a
matrix $C$ such that $P(Z)$ can be written as
\[
P(Z)=C\md(Z), \qquad P^\#(W)=\overline C\,\md(W).
\]
Put $J_N:=\operatorname{diag}(-1,1,\dots,1)$ and $\mathsf G:=C^T
J_N\overline C.$ Then
\[
B(P(Z),P^\#(W))=\md(Z)^T\mathsf G\md(W),
\]
and $\mathsf G$ is the Hermitian coefficient matrix of the polarized
form. 
\begin{lemma}\label{lem:coeffspace}
The restriction of the target Hermitian form to $\im \overline C\subset \CC^{N+1}$ is nondegenerate and has negative index one. Moreover, we have
\begin{equation}\label{eq:rangeeq}
\im(\mathsf G)=\im(C^T).
\end{equation}
\end{lemma}
\begin{proof}
Choose $Z$ in the negative cone $\{Z\in \mathbb{C}^{n+1}:\ H(Z)<0\}$
away from the polar locus. Then $Z_0\ne0$ and $z:=Z'/Z_0\in\B^n$.
Writing the affine lift as $P(1,z)=(q(z),p(z))$, we have $q(z)\ne0$ by holomorphicity in
the ball. From the degree-$m$ homogeneity of $P$,
we obtain
$$
P(Z)=Z_0^mP(1,z)=Z_0^m\big(q(z),p(z)\big).
$$
Together with $F(z)\in \BB^N$, it follows that
\[
B(P(Z),P^\#(\overline Z))
=|Z_0|^{2m}\bigl(-|q|^2+\|p(z)\|^2\bigr)=|Z_0|^{2m}|q(z)|^2\bigl(\|F(z)\|^2-1\bigr)<0.
\]
Thus $P(Z)=C\md(Z)$ contains a negative vector, whose complex
conjugate is a negative vector in $\im\overline C$.  In a Hermitian
space of signature $(N,1)$, the orthogonal complement of a negative
vector is positive definite.  Note that $\im\overline C$ contains a
negative vector and the ambient negative index is one; thus its
negative index is exactly one. Hence any vector in $\im\overline
C\cap(\im\overline C)^\perp$ is zero, so the restriction to
$\im\overline C$ is nondegenerate.

Next we show that $\ker\mathsf G=\ker\overline C$. In
fact, if $\ov{C}x=0,$ then $Gx=C^TJ_N\ov{C}x=0$, which means that
$\ker\overline C\subset \ker\mathsf G$.

On the other hand, if $\mathsf Gx=C^T J_N\overline{C}x=0$, then for
any $y\in \mathbb{C}^M$, where $M=\binom{n+m}{m}$, we have
$$
0=y^*Gx=\langle\ov{C}y,\ov{C}x\rangle_J.
$$
Since $\im\overline C=\{\ov{C}y:\ y\in \mathbb{C}^M\}$,
nondegeneracy of the restricted form to $\im\overline C$ gives
$\ov{C}x\in (\im\overline C)^\bot$. Thus $\ov{C}x\in \im\overline
C\cap(\im\overline C)^\bot=\{0\}$, which means $\overline Cx=0$.
Hence $\ker\mathsf G=\ker\overline{C}$. Then $\rk\mathsf
G=\rk \ov C=\rk C$. Since
\[
\im(\mathsf G)=\im(C^T
J_N\ov C)\subset\im(C^T)
\]
and both spaces have this same dimension, (\ref{eq:rangeeq})
follows.
\end{proof}

\begin{proposition}[Linear reduction]\label{prop:linear}
If a non-characteristic kernel direction exists, then $m=1$ and $F$
is totally geodesic.
\end{proposition}

\begin{proof}
If $m=1$ there is nothing to prove, so assume $m>1$.

Since $s=0$, Proposition~\ref{prop:global} gives
\[
\Div(Q|_{\cI})=p_1^*D_1+p_2^*D_2.
\]
Since $Q|_{\cI}$ has bidegree $(m-1,m-1)$, there are homogeneous polynomials
$\alpha\in\CC[Z]_{m-1}$ and $\beta\in\CC[W]_{m-1}$ such that
\[
Q|_{\cI}=c\,\alpha(Z)\beta(W)
\]
for some $c\in\CC^*$.  Hermitian symmetry $Q^\#=Q$ identifies the two factors,
and Lemma~\ref{lem:positivity} makes the remaining scalar positive.  After rescaling
$\alpha$, we therefore have
\begin{equation}\label{eq:nc-separation}
Q|_{\cI}=\alpha(Z)\alpha^\#(W).
\end{equation}
Then
\[
h\mid (Q-\alpha\alpha^\#).
\]
  We claim that $Q=\alpha\alpha^\#$.  Suppose otherwise and write it as
\[
Q=\alpha\alpha^\#+h^jR,\qquad j\ge1,\qquad h\nmid R.
\]
 Since every $Z$-derivative in $\Dop_\xi$ leaves the
factor $\alpha^\#(W)$ unchanged, there exists $N_{\xi}\in \kcone[W]$ such that
\[
\Dop_\xi (\alpha\alpha^\#)=\alpha^\#N_\xi.
\]
By Lemma~\ref{lem:product}, we obtain
\begin{equation}\label{eq:nc-leading}
\Dop_\xi(h^jR)\equiv-j(j+1)h^{j-1}\ell_v\ell_\xi R\pmod{h^j}.
\end{equation}
Since $\Dop_\xi Q=0$ by Lemma~\ref{lem:exactoperator}, equations
\eqref{eq:nc-leading} and $Q=\alpha\alpha^\#+h^jR$ give
\[
\alpha^\#N_\xi\equiv j(j+1)h^{j-1}\ell_v\ell_\xi R\pmod{h^j}.
\]
Over $\kcone[W]$ the polynomial $h$ is linear and hence prime, while
$h\nmid\alpha^\#$ because $\alpha^\#$ is independent of $Z$.  Thus
$h^{j-1}\mid N_\xi$; write $N_\xi=h^{j-1}M_\xi$.  Dividing by
$h^{j-1}$ and reducing modulo $h$ yields
\begin{equation}\label{eq:nc1}
j(j+1)\ell_v\ell_\xi R\equiv\alpha^\#M_\xi\pmod h.
\end{equation}

Now $\kcone[W]/(h)$ is a polynomial ring over the field $\kcone$, hence a UFD.
Non-characteristicity means precisely that $\ell_v\nmid\alpha^\#$ in the ring $\kcone[W]/(h)$.  Choose the constant vector $\xi$ generically so
that $\ell_\xi$ is also coprime to $\alpha^\#$ modulo $h$ and neither
$\ell_v$ nor $\ell_\xi$ vanishes modulo $h$.  

If $M_\xi=0$, then
\eqref{eq:nc1} gives $R\equiv0\pmod h$, contradicting $h\nmid R$.  
Hence 
$M_\xi\ne0$. Then
\eqref{eq:nc1} gives $\alpha^\#\mid R$ in the ring $\kcone[W]/(h)$, which is impossible because
\[
\deg_W\alpha^\#=m-1>m-1-j=\deg_WR.
\]
Hence $Q=\alpha\alpha^\#$, and therefore
\begin{equation}\label{eq:nc-kernelidentity}
B(P,P^\#)=h\,\alpha\alpha^\#.
\end{equation}

Now we use \eqref{eq:rangeeq} and \eqref{eq:nc-kernelidentity} to show that
\begin{align}\label{eq:nonred}
\alpha(Z)\mid P_j(Z)\qquad \text{for }j=0,1,\ldots,N.
\end{align}
Write
\[
B(P(Z),P^\#(W))=\sum_{|\beta|=m}R_\beta(Z)W^\beta,
\]
where each $R_\beta(Z)$ is homogeneous of degree $m$ in $Z$, and set
\[
\mathcal R_Z\bigl(B(P,P^\#)\bigr)
:=\operatorname{Span}_{\CC}\{R_\beta(Z):|\beta|=m\}
\subset \CC[Z]_m.
\]
Under the coefficient-vector identification determined by the monomial basis
$\md(Z)$, the coefficient vector of $R_\beta(Z)$ is the $\beta$-th column of
$\mathsf G$, while the coefficient vector of $P_j$ is the $j$-th row of $C$,
equivalently the $j$-th column of $C^T$.  Hence \eqref{eq:rangeeq} gives
\begin{equation}\label{eq:coefficient-span}
\mathcal R_Z\bigl(B(P,P^\#)\bigr)
=\operatorname{Span}_{\CC}\{P_0,\ldots,P_N\}.
\end{equation}
By \eqref{eq:nc-kernelidentity}, every coefficient $R_\beta(Z)$ is divisible by
$\alpha(Z)$.  Therefore \eqref{eq:coefficient-span} implies that every component
$P_j$ is divisible by $\alpha$, proving \eqref{eq:nonred}.  Since $P$ is reduced,
$\alpha$ must be constant.  But $\deg\alpha=m-1>0$, a contradiction.  Hence
$m=1$, and the map is totally geodesic.
\end{proof}

\subsection{The pure characteristic branch}

Henceforth assume $F$ is nonlinear.  Proposition~\ref{prop:linear}
removes every non-characteristic kernel direction.
Proposition~\ref{prop:rge2} therefore gives $r=1$, and
Proposition~\ref{prop:nomixed} leaves only the pure characteristic
branch.

Let
\[
\Delta_\kappa=p_2^{-1}(\Pi_L), \qquad \Pi_L=\{L(W)=0\}.
\]
Choose $e\in\CC^{n+1}$ such that $L(W)=h(e,W)$.

\begin{lemma}\label{lem:constantv}
After shrinking to a nonempty Zariski open set and rescaling the
kernel line field, the pure characteristic direction admits the
constant cone representative
\[
v(Z)\equiv e\pmod{\CC Z}.
\]
\end{lemma}

\begin{proof}
For general $Z$, the characteristic hyperplane cut out on the
incidence fiber by $h(v(Z),W)$ is the fixed hyperplane cut out by
$L(W)=h(e,W)$.  Thus the restrictions of the two linear forms to
$\{W:h(Z,W)=0\}$ are proportional.  Nondegeneracy of $h$ gives
\[
v(Z)=c(Z)e+g(Z)Z
\]
with $c(Z)\ne0$ generically.  Replacing the generator by
$c(Z)^{-1}v(Z)$ and then changing it by the radial term
$-c(Z)^{-1}g(Z)Z$ leaves the same projective tangent direction and
gives the constant representative $e$.
\end{proof}

\begin{lemma}\label{lem:center}
One has $H(e)<0$.  After a source ball automorphism and a nonzero
scalar rescaling of the kernel representative, one may take
\[
e=e_0:=(1,0,\ldots,0).
\]
\end{lemma}

\begin{proof}
Suppose to the contrary that $H(e)\ge0$. We first prove there is a nonzero $\omega\in\CC^{n+1}$ such that 
\[
H(\omega)=0,\qquad h(e,\overline{\omega})=0.
\]
If $H(e)=0$, take $\omega=e$.
If $H(e)>0$, the Hermitian orthogonal complement
\[
e^{\perp_H}:=\{Z\in\CC^{n+1}:h(e,\overline Z)=0\}
\]
has signature $(n-1,1)$, and hence contains a nonzero null vector
$\omega$.  

Therefore $([\omega],[\overline{\omega}])\in\cI$ and
$L(\overline{\omega})=h(e,\overline{\omega})=0$. Therefore by definition
 \[([\omega],[\overline{\omega}])\in\Delta_\kappa
 \]
 Since $s=\ord_{\Delta_\kappa}(Q|_{\cI})>0$, we have
$\Delta_\kappa\subset\Div(Q|_{\cI})$.  It follows that
\[
Q(\omega,\overline{\omega})=0,
\]
contradicting Lemma~\ref{lem:positivity}.  Thus $H(e)<0$.

The group $U(n,1)$ acts transitively on negative lines.  Hence a
source ball automorphism sends $[e]$ to $[e_0]$; rescaling the
chosen generator of the kernel line field then gives $e=e_0$.
\end{proof}

\begin{lemma}\label{lem:pureradial}
After the normalization of Lemma~\ref{lem:center} and a target ball
automorphism sending $F(0)$ to $0$, the image under $F$ of a general
complex radial disc is contained in a complex line through the
origin.  More precisely, for general
$\zeta\in\partial\mathbb B^n$ there is a unit vector
$u_\zeta\in\partial\mathbb B^N$ such that
\[
F(t\zeta)\in\CC u_\zeta,\qquad |t|<1.
\]
\end{lemma}

\begin{proof}
Source and target ball automorphisms are induced by projective linear
transformations preserving the corresponding Hermitian forms.  Thus
they preserve the generic projective Gauss defect and the degree of a
reduced homogeneous lift.  We therefore continue to write $F$ and
$P$ for the normalized map and its reduced homogeneous lift.

For a general nonzero $z\in\CC^n$, consider the projective line
\[
\ell_z:=\{[1:tz]:t\in\CC\}\cup\{[0:z]\}\subset\PP^n.
\]
At its general points, $\ell_z$ is an integral curve of the
Gauss-kernel field $e_0$.  Indeed, for $Z_t=(1,tz)$ with $t\ne0$,
\[
\dot Z_t=(0,z),\qquad e_0-Z_t=-t\dot Z_t,
\]
so the classes of $\dot Z_t$ and $e_0$ span the same line in
$T_{[Z_t]}\PP^n\simeq\CC^{n+1}/\CC Z_t$.  By
Lemma~\ref{lem:constantv}, this is exactly the Gauss-kernel direction.
Hence, at every general point of $\ell_z$,
\[
\frac{d}{dt}\Gamma_{\rm proj}([1:tz])=0.
\]
The restriction of the rational map $\Gamma_{\rm proj}$ to
$\ell_z$ is therefore constant on a nonempty open subset.

On the open set where $[P]$ is immersive and its image lies in
$X_{\rm reg}$, equation~\eqref{eq:Gamma-factor} gives
$\Gamma_{\rm proj}=\Gamma_X\circ[P]$.  Since the Gauss defect is one,
a general fiber of $\Gamma_X$ has dimension one.  Theorem~\ref{thm:GH}
then shows that the closure of such a fiber is a projective line,
which we denote by $\Lambda_z\subset\PP^N$.  Thus
\[
[P(1,tz)]\in\Lambda_z
\]
for $t$ in a nonempty open subset, and hence, by closure, wherever the
restriction is defined.  In particular, $[P(1,0)]\in\Lambda_z$.
Because the target normalization gives $F(0)=0$, the point
$[P(1,0)]$ is the target origin in projective coordinates.  Therefore
the affine part of $\Lambda_z$ is a complex line through $0$.
Taking $z=\zeta$ in a general boundary direction and normalizing a
direction vector of that line to have norm one gives the stated
vector $u_\zeta$.
\end{proof}

\begin{proposition}[Pure-branch reduction]\label{prop:terminal}
In the pure characteristic branch, after source and target ball
automorphisms and zero padding,
\[
F=V_m\oplus0.
\]
\end{proposition}

\begin{proof}
Keep the normalization in Lemma~\ref{lem:pureradial}, so that
$F(0)=0$.  Let $r\ge1$ be the vanishing order of $F$ at the origin,
and write its Taylor expansion as
\[
F(z)=F_r(z)+F_{r+1}(z)+\cdots,
\]
where each $F_\nu$ is homogeneous of degree $\nu$ and
$F_r\not\equiv0$. 

Applying the maximal principle, $F_r(\zeta)\ne0$ for general $\zeta\in\partial\mathbb B^n$. Hence we can choose a general $\zeta\in\partial\mathbb B^n$ for which
Lemma~\ref{lem:pureradial} applies and $F_r(\zeta)\ne0$.  With the
unit vector $u_\zeta$ from that lemma, there is a holomorphic
function $b_\zeta$ on a neighborhood of the closed unit disc such
that
\begin{equation}\label{eq:radial-b}
F(t\zeta)=b_\zeta(t)u_\zeta,\qquad |t|\le1.
\end{equation}
Since $F$ maps the unit sphere to the unit sphere,
\[
|b_\zeta(t)|=1\qquad (|t|=1).
\]
Thus $b_\zeta$ is a finite Blaschke product.

We next choose $\zeta$ so that the radial disc contains no nonzero
point of $F^{-1}(0)$.  This is possible for a general $\zeta$:
properness makes $F^{-1}(0)$ compact, while it is an analytic subset
of the Stein manifold $\mathbb B^n$, and hence it is finite.  We need
only avoid the finitely many radial directions determined by the
nonzero points of this fiber.  For such a $\zeta$, the only zero of
$b_\zeta$ in the unit disc is $0$.  Its multiplicity there is exactly
$r$, because $F_r(\zeta)\ne0$.  Consequently
\[
b_\zeta(t)=c_\zeta t^r,\qquad |c_\zeta|=1.
\]
Evaluating at $t=1$ in \eqref{eq:radial-b} gives
$c_\zeta u_\zeta=F(\zeta)$, and therefore
\begin{equation}\label{eq:radial-homogeneity}
F(t\zeta)=t^rF(\zeta)
\end{equation}
for every $|t|<1$ and general $\zeta\in\partial\mathbb B^n$.

Comparing the powers of $t$ in the Taylor expansion of
\eqref{eq:radial-homogeneity}, we obtain
\[
F_\nu(\zeta)=0\qquad\text{for every }\nu\ne r
\]
for general $\zeta$ on the sphere.  Such boundary directions are
Zariski dense in $\CC^n$, so $F_\nu\equiv0$ for every $\nu\ne r$.
Hence $F=F_r$ is a homogeneous polynomial map of degree $r$.

The normalized map is obtained from the original one by source and
target projective linear transformations, so the degree $m$ of the
reduced homogeneous lift is unchanged.  On the other hand, the
homogeneous polynomial map $F_r$ has the reduced homogeneous lift
\[
\widetilde P(Z_0,Z')=(Z_0^r,F_r(Z')).
\]
Thus the projective degree is $r$, and therefore $r=m$.

For $\|\zeta\|=1$, properness gives $\|F(\zeta)\|=1$.  Homogeneity
and $r=m$ now imply
\[
\|F(z)\|^2=\|z\|^{2m}\qquad (z\in\CC^n).
\]
Polarizing this identity yields
\begin{equation}\label{eq:polarized-veronese}
\langle F(z),F(w)\rangle=\langle z,w\rangle^m.
\end{equation}
Because the vectors $V_m(z)$ span $\CC^{M_m}$, we can write $F(z)=A V_m(z)$ for a linear map $A:\CC^{M_m}\to\CC^N$.  By the defining
normalization of the Veronese map,
\[
\langle V_m(z),V_m(w)\rangle=\langle z,w\rangle^m.
\]
Equation~\eqref{eq:polarized-veronese} therefore gives
$A^*A=I_{M_m}$. After some
unitary change of target coordinates, we have $A=\begin{pmatrix}I_{M_m}\\0\end{pmatrix}.$ Consequently $F=V_m\oplus0$ as claimed.
\end{proof}

\subsection{Proof of the main theorem}

\begin{proof}[Proof of Theorem~\ref{thm:main}]
By Proposition~\ref{prop:bridge}, the projective Gauss map
has positive generic defect.  The preceding sections produce the
canonical fixed operator and prove the residual logarithmic
equation, the qualified tangency of dominating residual components,
and global divisor separation without any target-dimension bound.

If the Gauss defect is at least two, Proposition~\ref{prop:rge2}
gives a non-characteristic kernel direction.  If the defect is one
and the characteristic divisor is mixed,
Proposition~\ref{prop:nomixed} gives a contradiction; its proof uses
the classical Gauss-fiber theorem stated as Theorem~\ref{thm:GH}.
Therefore either a non-characteristic direction exists or the map
lies in the defect-one pure branch.

In the non-characteristic case, Proposition~\ref{prop:linear} gives
$m=1$ and total geodesicity.  In the pure branch,
Lemma~\ref{lem:center} shows that the fixed characteristic center is
negative.  After moving it to the origin, Lemma~\ref{lem:constantv}
identifies the Gauss-kernel direction with the affine radial
direction.  The Griffiths--Harris theorem then forces the image of a
general radial line to lie in a projective line through the target
origin.  Properness reduces the corresponding one-variable maps to
monomials, and Proposition~\ref{prop:terminal} concludes that $F$ is
holomorphically equivalent to $V_m\oplus0.$
\end{proof}

\vspace{.5cm}

\noindent\textbf{Acknowledgments:} The second author would like to thank Xiaojun Huang for bringing this problem to his attention in 2009. He is also grateful to Xiaojun Huang, Shanyu Ji, Yuan Zhang, and Yuan Yuan for helpful discussions concerning this problem. During the preparation of this work, the authors used ChatGPT as an auxiliary tool. The authors verified all mathematical arguments independently and take full responsibility for the content of this paper.

\end{document}